\documentclass[12pt]{amsart}
\usepackage{amssymb,amsfonts,amsmath,amsopn,amstext,amscd,latexsym, amsthm, enumerate,mathrsfs}

\usepackage{color}
\usepackage{xcolor}
\usepackage[]{hyperref}
\hypersetup{
     colorlinks   = false,
     citecolor    = blue
}

\usepackage[margin=30mm]{geometry}
\usepackage{bbm}
\usepackage[all,cmtip]{xy}

\usepackage{enumerate}
\usepackage[shortlabels]{enumitem}

\newtheorem{theorem}{Theorem}[section]

\newtheorem{lemma}[theorem]{Lemma}

\theoremstyle{definition}

\theoremstyle{remark}

\DeclareMathOperator{\Aut}{Aut}

\DeclareMathOperator{\Irr}{Irr}

\DeclareMathOperator{\lcm}{lcm}

\DeclareMathOperator{\SL}{SL}

\newcommand{\cA}{{\mathcal A}}
\newcommand{\cB}{{\mathcal B}}

\DeclareMathOperator{\Sym}{S}

\DeclareMathOperator{\PSL}{PSL}
\DeclareMathOperator{\PSU}{PSU}

\DeclareMathOperator{\PGL}{PGL}

\newcommand{\cd}{{\mathrm {cd}}}

\newcommand{\Centralizer}{\mathbf{C}}

\numberwithin{equation}{section}

\newcommand{\Alt}{{\mathrm {A}}}

\newcommand{\Sz}{{\mathrm {Sz}}}

\begin{document}

\title[two-prime hypothesis]{Groups satisfying the two-prime hypothesis with a composition factor isomorphic to $\PSL_2(q)$ for $q\geq 7$}

\author[M. L. Lewis]{Mark L. Lewis}
\address{Department of Mathematical Sciences, Kent State University, Kent, OH 44242, USA}
\email{lewis@math.kent.edu}

\author{Yanjun Liu}
\address{
College of Mathematics and Information Science, Jiangxi Normal University, Nanchang, $330022$, P.R. China}
\email{liuyj@math.pku.edu.cn  }

\author[H. P. Tong-Viet]{Hung P. Tong-Viet}
\address{Department of Mathematical Sciences, Kent State University, Kent, OH 44242, USA}
\email{htongvie@math.kent.edu}

\begin{abstract} Let $G$ be a finite group, and write $\cd(G)$ for the degree set of the complex irreducible characters of $G$.  The group $G$ is said to satisfy the {\it two-prime hypothesis} if, for any distinct degrees $a, b \in \cd(G)$, the total number of (not necessarily different) primes of the greatest common divisor $\gcd(a, b)$ is at most $2$.
In this paper, we prove an upper bound on the number of irreducible character degrees of a nonsolvable group that has a composition factor isomorphic to $\PSL_2 (q)$ for $q \geq 7$.
\end{abstract}

\thanks{}



\date{January 11, 2017}


\maketitle

\section{Introduction}

Throughout this paper, $G$ will be a finite group, $\Irr(G)$ will be the set of complex irreducible characters of $G$, and $\cd(G) = \{ \chi (1) \mid \chi \in \Irr(G) \}$ will be the degree set of $\Irr(G)$.  Following \cite{Hamblin}, we will say that a group $G$ satisfies the {\it two-prime hypothesis} if whenever $a, b \in \cd(G)$ with $a \ne b$, then $\gcd (a,b)$ is divisible by at most two primes counting multiplicity.  Solvable groups satisfying the two-prime hypothesis were studied in \cite{Hamblin, HamblinL} where an upper bound was determined for the number of character degrees for such groups.

In \cite{LL}, the first two authors started studying nonsolvable groups that satisfy the two-prime hypothesis.  In that paper, we determined that any nonabelian chief factor of a group satisfying the two-prime hypothesis will be simple or isomorphic to $\Alt_5 \times \Alt_5$.  They also determined which nonabelian simple groups can occur as the chief factor of a group satisfying the two-prime hypothesis.

In this paper, we continue to study the nonsolvable groups that satisfy the two-prime hypothesis, and in particular, we focus on bounding the number of character degrees of such a group.  In \cite{LLT}, we showed when $G$ is a nonsolvable group satisfying the two-prime hypothesis where none of the composition factors is isomorphic to any $\PSL_2 (q)$ for a prime power $q$ that $|\cd (G)| \le 21$.  We also gave an example of a group $G$ satisfying the two-prime hypothesis where this bound was met.  We note that the group $G$ had one nonabelian composition factor and that composition factor was isomorphic to $\rm{A}_7$.

The main purpose of this second paper is to give an upper bound for $|\cd(G)|$ when $G$ is a nonsolvable group satisfying the two-prime hypothesis that has a nonabelian composition factor that is isomorphic to $\PSL_2 (q)$ for any prime power $q$.

\begin{theorem}\label{th:thispaper}
Let $G$ be a group satisfying the two-prime hypothesis.  If $G$ has a composition factor isomorphic to $\PSL_2 (q)$ for some prime power $q \ge 7$, then $|\cd(G)| \leq 19$.
\end{theorem}

Note that the only composition factor not covered by Theorem \ref{th:thispaper} and Theorem 1.1 of \cite{LLT} is $\Alt_5$, so when we combine Theorem \ref{th:thispaper} with Theorem 1.1 of \cite{LLT}, we obtain the following theorem.

\begin{theorem}\label{th:main}
Let  $G$ be a nonsolvable group satisfying the two-prime hypothesis.  If $G$ has no nonabelian composition factors isomorphic to $\Alt_5$, then $|\cd(G)| \leq 21$.
\end{theorem}

Here we use the bound in Theorem \ref{th:main} that was found in \cite{LLT}, since there is an example of a nonsolvable group satisfying the two prime hypothesis that has $21$ character degrees and we believe that $21$ is likely the best bound.

To prove Theorem \ref{th:main}, we first determine for which almost simple groups that have a composition factor of a group that is isomorphic to $\PSL_2 (q)$ satisfying the two-prime hypothesis.  We then see that every group that has a composition factor isomorphic to $\PSL_2 (q)$ has a quotient that is one of the above almost simple groups.  We then break up the possible character degrees that can occur into three different subsets, and we obtain a bound on each of these possible subsets when $q \ge 7$.  We make strong use of the fact that two of $q-1$, $q$, and $q+1$ are divisible by many primes, and we will need to pay special attention to the cases where $q$ is small and the number of prime divisors of these integers are fewer.

In \cite{LLT}, we handled the groups satisfying the two prime hypothesis that had a nonabelian composition factor that is not isomorphic to some $\PSL_2 (q)$.  In that paper, we were easily able to handle the groups having those composition factors that are not isomorphic to $A_7$ just using the character degrees of the simple group, degrees of the projective characters, and the indices of the maximal subgroups.  In those cases, there were character degrees and projective character degrees that were divisible by at least three primes, and there were relatively few subgroups whose indices were not divisible by three primes.

Nearly half of \cite{LLT} dealt with the case where the composition factor is $A_7$.  That case is more difficult since all of the character degrees of $A_7$ are divisible by only two primes, the representation group for $A_7$ is more complicated and most of the projective character degrees are divisible by only two primes, and there are several subgroups whose indices are divisible by only one or two primes.  We end up breaking the set of character degrees into five subsets and bounding the sizes of each of those subsets.

The difficulty that arose in the case of $A_7$ also arises when bounding the number of character degrees of a group satisfying the two prime hypothesis with a composition factor that is isomorphic to $A_5$.   Although $A_5\cong \PSL_2(4)$ is smallest in the series $\PSL_2(q)$, the situation appears to be the most difficult.  The first difficulty is that there can be two composition factors that are isomorphic to $A_5$, however that case seems to be relatively easy to dispatch with.  The real difficulty is when there is only one nonabelian composition factor.  The difficulty arises due to the fact that two of the character degrees of $A_5$ are primes and the third degree is the square of a third prime, so all the degrees are relatively primes.  Also, there are many subgroups whose indices are divisible by only one or two primes. All of these make many of the arguments in \cite{LLT} and this paper invalid when dealing with $A_5$.

In fact, to get a bound for the case for $A_5$, we essentially find pairs of character degrees of some Frobenius groups that always occur simultaneously.  Since the proof is quite different from and longer than the proof contained in this paper, we prefer to dealing with it in a future paper.

The notation used in this paper is mostly standard (cf. \cite{Atlas} and \cite{Isaacs}).  We now fix some notation that will be used throughout the paper.  Let $N$ be a normal subgroup of $G$ and let $\theta \in \Irr(N)$. As usual, $I_G(\theta)$ means the inertia group of $\theta$ in $G$, we write $\Irr (G| \theta)$ for the set of irreducible constituents of $\theta^G$, and $\cd (G |\theta) = \{ \chi (1) \mid \chi \in \Irr(G| \theta) \}$.  If $n>1$ is an integer, then $\omega(n)$ denotes the number of prime divisors of $n$ counting multiplicity.

\section{Preliminaries}


In \cite{LL}, the first two authors identified the nonabelian chief factors of nonsolvable groups that satisfy the two-prime hypothesis.  The results are encoded in the following theorem.

\begin{theorem} \label{chief}
Let $G$ be a nonsolvable group satisfying the two-prime hypothesis.  Suppose that $M/N$ is a nonabelian chief factor of $G$.

\begin{enumerate}
  \item If $M/N$ is simple, then it is one of the following groups: $\Alt_5$, $\Alt_6$, $\Alt_7$, $\PSL_2(q)$ for $q\geq 7$, $\PSL_3 (3)$, $\PSL_3 (4)$, $\PSL_3 (5)$, $\PSU_3 (3)$, $\PSU_3 (4)$, $\Sz (8)$, $\Sz (32)$, and ${\rm M}_{11}$.
  \item If $M/N$ is not simple, then $M/N \cong \Alt_5\times \Alt_5$.
\end{enumerate}
\end{theorem}

Let $G$ be a finite nonsolvable group satisfying the two-prime hypothesis.  Using the notation in Theorem {\rm \ref{chief}} and its proof in \cite{LL}, we find a normal subgroup $N$ so that $(G/N)/(CN/N)$ embeds into $\Aut( {M/N})$ where $C/N = \Centralizer_{G/N} (M/N)$, which is almost simple except when $M/N \cong \Alt_5 \times \Alt_5$.  Consequently,  we can conduct a case-by-case investigation on the possible nonabelian chief factors $M/N$ of $G$ in order to obtain a general upper bound for $|\cd(G)|$.  In the present paper, we get an upper bound for  $|\cd(G)|$ of nonsolvable groups $G$ with  nonabelian chief or composition factors isomorphic to $\PSL_2(q)$ for any prime power $q\ge 7$.

In this section, we consider almost simple groups with  socle isomorphic to $S \cong \PSL_2 (q)$ for a prime power $q \geq 7$.  Let $q = p^f$, where $p$ is a prime and $f$ is a positive integer.  The outer automorphism group of $S$ is generated by a field automorphism $\varphi$ of order $f$, and if $p$ is odd, a diagonal automorphism $\overline{\delta}$ of order $2$. For any groups with socle $\PSL_2(q)$, the character degrees have been explicitly computed by D. White.

\begin{lemma} \emph{(\cite[Theorem A]{White}).} \label{lem:White}
Let $S = \PSL_2 (q)$, where $q = p^f > 3$ for a prime $p$ and a positive integer $f$, let $A = \Aut(S)$, and let $S \leq H \leq A$.  Set $\Gamma = {\rm PGL}_2 (q)$ if $p$ is odd and $\overline {\delta} \in H$ and $\Gamma = S$ if either $p = 2$ or $\overline {\delta} \not\in H$, and set $d = |H:\Gamma| = 2^a m$, for the nonnegative integer $a$ and the odd integer $m$. If $p$ is odd, let $\epsilon = (-1)^{\frac {(q-1)}{2}}$.  Then the set of irreducible character degrees of $H$ is $$\cd(H) = \{ 1, q, (q+\epsilon)/2 \} \cup \{ (q-1) 2^a i : i \mid m \} \cup \{ (q+1) j : j \mid d \},$$ with the following exceptions:
\begin{itemize}
   \item[$(i)$]   If $p$ is odd with $H \nleqslant S \langle \varphi \rangle$ or if $p=2$, then $(q + \epsilon)/2$ is not a degree of $H$.
   \item[$(ii)$]  If $f$ is odd, $p = 3$, and $H = S \langle \varphi \rangle$, then $i \neq 1$.
   \item[$(iii)$] If $f$ is odd, $p = 3$, and $H = A$, then $j\neq 1$.
   \item[$(iv)$]  If $f$ is odd, $p \in \{ 2, 3, 5 \}$ and $H = S \langle \varphi \rangle$, then $j \neq 1$.
   \item[$(v)$]   If $f \equiv 2 ~({\rm mod ~}4)$, $p \in \{ 2, 3 \}$, and $H = S \langle \varphi \rangle$ or $H = S \langle \overline{\delta} \varphi \rangle$, then $j \neq 2$.
\end{itemize}
\end{lemma}

We mention that when $p = 3$, the conditions of exceptions $(ii)$ and $(iv)$ are the same.  Therefore, if this is the case, then we have both $i \neq 1$ and $j \neq 1$.

\begin{theorem} \label{AlmostL2(q)}
Let $S < G \leq \Aut(S)$, where $S = \PSL_2 (q)$ for $q = p^f \geq 7$. If $G$ satisfies the two-prime hypothesis, then $G$ is one of the groups in Table \ref{tab:L2}.

\begin{table}
{\small
\begin{center}
\caption{Possible group, character degrees, and necessary conditions}\label{tab:L2}
 \vspace{-2ex}
\begin{tabular}{lll}
  \hline

 $G$ & $\cd(G)\setminus \{1\}$ & {\rm Conditions} \\ \hline
 $\Sym_6$ & $\{5, 9, 10, 16\}$      & \\
 ${\rm{M}}_{10}$ & $\{ 9, 10 ,16\}$      &  \\
 $\PGL_2(q)$ & $\{q-1, q, q+1\}$    &   \\
 $ S \langle \varphi \rangle$ & $\{(3^f-1)/2, 3^f, (3^f\pm1)f\}$ &  $\omega((3^f-1)/2)\leq 2$ {\rm and}\\
  & &{\rm $f$  is an odd prime} \\
 $ \Aut(S)$  &  $\{ 3^f-1, 3^f, (3^f\pm1)f \}$ & $\omega(3^f-1) = 2$ {\rm and}  \\
 & & $f$ {\rm is an odd prime} \\
 $ S \langle \varphi \rangle$ & $\{ 2^f-1, 2^f, (2^f-1)f, (2^f+1)f \}$ & $\omega(2^f-1)\leq 2$ {\rm and} \\
 &                             &    $f$ {\rm is an odd prime}\\
$S\langle \varphi^{\frac{f}{4}}\rangle$  & $\{2^f, 2^2(2^f\pm1), 2^f+1, 2(2^f+1)\}$ & $2^f+1 > 5$ {\rm is a Fermat prime}   \\
$S\langle \varphi^{\frac{f}{2}}\rangle$  & $\{2^f,  2(2^f-1), 2^f+1, 2(2^f+1)\}$ &  $\omega(2^f+1)\leq 2$ {\rm and} $f$ {\rm is even}  \\
$S\langle \varphi^{\frac{f}{2}}\rangle$  & $\{ \frac{q+1}{2}, q,  2(q-1), q+1, 2(q+1)\}$  &  $\omega(q+1)\leq 2$, \\
&                          & {\rm $f$ is even, and  $q$ is odd} \\
$S\langle \overline{\delta}\varphi^{\frac{f}{2}}\rangle$  & $\{ q,  2(q-1), q+1, 2(q+1)\}$  &  $\omega(q+1)\leq 2$, \\
&                          & {\rm $f$ is even, and  $q$ is odd} \\

$\PGL_2(q)\langle\varphi^{\frac{f}{2}}\rangle$  & $\{q, 2(q-1), q+1, 2(q+1)\}$ &  $\omega(q+1)\leq 2$, \\
&                          & {\rm  $f$ is even, and $q$ is odd}\\
 $S\langle\varphi^{\frac{f}{m}}\rangle$ & $\{2^f-1, 2^f, m(2^f\pm1), 2^f+1\}$ & $\omega(2^f-1)\leq 2$,$\omega(2^f+1)\leq 2$ {\rm and}\\
 &   & {\rm $m<f$ is an odd prime dividing} $f$   \\
\hline
\end{tabular}
\end{center}
}
\end{table}
\end{theorem}

\begin{proof}
We use notation in Lemma \ref{lem:White}.  Suppose $G = {\PGL}_2 (q)$.  Since $\cd(\PGL_2 (q)) = \{ 1, q-1, q, q+1\}$, it is easy to see that $G$ satisfies the two-prime hypothesis.  So we may assume that $G \not\cong {\PGL}_2(q)$.  Then we have $a \geq 1$ when $m = 1$.  We will first deal with the exceptions in Lemma \ref{lem:White}, and then we will deal with the general case.

We consider the exceptions $(ii)$, $(iii)$, $(iv)$, and $(v)$ of Lemma \ref{lem:White} separately.   In these exceptions, we have $p \in \{ 2, 3, 5 \}$.  So we may consider the possibilities for $p$.  We will deal with exception $(i)$ as part of the general case.

(\textbf{1}) $p = 5$.  In this case, we only need to consider exception $(iv)$ of Lemma \ref{lem:White}.

We have $G = S \langle \varphi \rangle$, $m = f>1$ is odd, and $5^f - 1, (5^f - 1)f \in \cd(G)$.  Notice that $5^f$ is congruent to $1$ modulo $4$, so $4$ divides $5^f - 1$.  Since $f > 1$, we have $4 < 5^f - 1$ and this implies $5^f-1$ is divisible by at least three primes.  Thus, $G$ does not satisfy the two-prime hypothesis, so this case does not occur.

(\textbf{2}) $p = 3$. If exception $(ii)$ occurs in Lemma \ref{lem:White}, then exception $(iv)$ also occurs.  We have $G = S \langle \varphi \rangle$, $m = f > 1$ is odd, and $\cd(G) = \{ 1, (3^f-1)/2, 3^f, (3^f-1)f, (3^f+1)f \}$.  Since $G$ satisfies the two-prime hypothesis and $2f$ divides both $(3^f - 1) f$ and $(3^f + 1)f$, it follows that $f$ is an odd prime.  Also, we must have $(3^f - 1)/2$ is divisible by at most two primes, as listed in Table \ref{tab:L2}.

Suppose that exception $(iii)$ occurs.  Then $G= \Aut(S)$ and $m = f$ is odd.  We observe that both $3^f-1$ and $3^f+1$ are divisible by at least two primes.   Since $G$ satisfies the two-prime hypothesis, $f$ must be an odd prime and $\cd(G)=\{1, (3^f-1), 3^f, (3^f-1)f, (3^f+1)f \}$ by Lemma \ref{lem:White}.  In addition, we see that $3^f-1$ will be divisible by exactly two primes, as listed in Table \ref{tab:L2}.

We now suppose that exception $(v)$ occurs.  Then $f \equiv 2 ~({\rm mod~} 4)$, and $G= S \langle \varphi \rangle$ or $G = S \langle \overline{\delta} \varphi\rangle$.  If $f = 2$, then we have $G = \Sym_6$ or $H = \textrm{M}_{10}$, which are contained in Table \ref{tab:L2}.  If $f > 2$, then $2 (3^f-1), (3^f-1) f \in \cd(G)$.  Observe that $3^f$ is congruent to $1$ modulo $8$ since $f$ is even, and so, $8$ divides $3^f - 1$.  This implies $G$ does not satisfy the two-prime hypothesis and this case cannot occur.

(\textbf{3}) $p = 2$.  We need to consider exceptions $(iv)$ and $(v)$ in Lemma \ref{lem:White}.  We also know the exception $(i)$ of Lemma \ref{lem:White} occurs, and so, we have that $\frac {2^f + \epsilon}{2}$ is not a degree of $G$.  Suppose that exception $(iv)$ occurs.  Then $G = S \langle \varphi \rangle$,  $m = f$ is odd, and $2^f - 1, (2^f - 1) f \in \cd(G)$.  Since $G$ satisfies the two-prime hypothesis, we have $\omega (2^f - 1) \leq 2$.  We claim that this implies that $f$ is a prime or the square of a prime.  To do this, we apply the Zsigmondy prime theorem which can be found as Theorem IX.8.3 in \cite{HuppertB}.  Suppose that $f$ is divisible by two distinct primes $p_1$ and $p_2$, then since $p_1 p_2$ is odd, we get Zsigmondy prime divisors for each of $2^{p_1} - 1$, $2^{p_2} - 1$, and $2^{p_1 p_2} - 1$, and this gives three distinct prime divisors of $2^f - 1$, a contradiction.  Thus, we have that $f$ is a power of some prime $r$.  If $r^3$ divides $f$, then we get Zsigmondy prime divisors for each of $2^r - 1$, $2^{r^2} - 1$, and $2^{r^3} - 1$, and again this gives three distinct prime divisors of $2^f - 1$.  It follows that $f$ is either $r$ or $r^2$, as claimed.
Assume $f = r^2$.  This implies that $2^f-1$ is not a Mersenne prime, and so $2^f - 1$ has exactly two prime divisors.
Also, we have $(2^f - 1)r, (2^f - 1)r^2 \in \cd(G)$, contradicting the two-prime hypothesis. Therefore, this cannot occur, and thus, $f$ is an odd prime
and $\cd(G)=\{ 1,  2^f-1, 2^f, (2^f-1)f, (2^f+1)f \}$, where $2^f - 1$ is divisible by at most two primes, as listed in Table \ref{tab:L2}.

Now, suppose that exception $(v)$ occurs.  We have $f > 2$. Then we have $m > 1$ and $2(2^f-1), (2^f-1)f \in \cd(G)$.  Since $f$ is congruent to $2$ modulo $4$, $f$ is even and $2^f$ is congruent to $1$ modulo $3$.  This implies that $3$ divides $2^f - 1$.  Since $f > 2$, we know that $3 < 2^f - 1$ and $2^f - 1$ is not prime.  Since $f$ is even, we see that $\gcd (2 (2^f-1),f(2^f-1)) = 2(2^f-1)$ which is divisible by at least three primes, so $G$ does not satisfy the two-prime hypothesis.  Therefore, this case cannot occur.

We now consider the general case of Lemma \ref{lem:White}.

(\textbf{4}) $a \geq 2$.
By Lemma \ref{lem:White}, we have $(q - 1) 2^a, (q + 1) 2^a \in \cd(G)$ if either $a \geq 3$ or $a = 2$, $m = 1$ and $q$ is odd, and that $(q-1) 2^a, (q-1) 2^a m \in \cd(G)$ if $a = 2$ and $m > 1$.  In all of these cases, the two-prime hypothesis is violated in $G$, so they do not occur.  Suppose that $a=2$, $m=1$, and $q$ is even.  Then we have $\cd(G) = \{ 1, 2^f, 2^2 (2^f-1), (2^f+1), 2(2^f+1), 2^2(2^f+1) \}$ and $G = S \langle \varphi^{\frac{f}{4}} \rangle$.  Since $G$ satisfies the two-prime hypothesis, $2^f+1$ must be a (Fermat) prime, as listed in Table \ref{tab:L2}.

(\textbf{5}) $a = 1$. In particular, $f$ is even.
If $m > 1$, then $\cd(G) \supseteq \{2(q-1),  2m(q-1), 2(q+1), 2m(q+1)\}$.  Since $q > 5$, it follows that $q - 1$ and $q + 1$ cannot simultaneously be a Mersenne prime and a Fermat prime. Thus, either $2(q-1)$ or $2(q+1)$ is divisible by at least three primes, and this violates the two-prime hypothesis, so this case cannot occur.  We now assume $m = 1$. If  $q$ is even then  $G= S \langle \varphi^{\frac{f}{2}}\rangle $ with  $\cd(G)=\{1, 2^f, 2(2^f-1), 2^f+1, 2(2^f+1) \}$;
if $q$ is odd and $G\leq S \langle \varphi\rangle $
then  $G= S \langle \varphi^{\frac{f}{2}}\rangle $ with  $\cd(G)=\{1, \frac{q+1}{2}, q, 2(q-1), q+1, 2(q+1) \}$; and if
$q$ is odd and $G\not\leq S \langle \varphi\rangle $ then $G= {\PGL}_2(q)\langle \varphi^{\frac{f}{2}} \rangle$
with $\cd(G)=\{1, q, 2(q-1), q+1, 2(q+1) \}$ or $G=\PSL_2(q)\langle\bar{\delta}\varphi^{f/2}\rangle$ with $\cd(G)=\{1,q,q+ 1,2(q\pm 1)\}$.
Since $G$ satisfies the two-prime hypothesis, it follows that $q+1$ is divisible by at most two primes for every case, as listed in Table \ref{tab:L2}.

(\textbf{6}) $a = 0$ so that $m > 1$.  If $q$ is odd, then as we have seen either $q - 1$ or $q + 1$ is divisible by at least three primes since $q > 5$. By Lemma \ref{lem:White}, we have $\{ q - 1, (q - 1) m, q + 1, (q + 1) m \} \subset \cd(G)$.  This implies that $G$ does not satisfy the two-prime hypothesis, so this case does not occur.  Suppose that $q$ is even. If $m$ is composite with a proper divisor $m_1$, then $\{ (2^f - 1) m_1, (2^f - 1) m, (2^f + 1) m_1, (2^f + 1) m \}\subset \cd(G)$ by  Lemma \ref{lem:White}.  Observe that $2^f$ is congruent to $1$ modulo $3$ when $f$ is even and $-1$ when $f$ is odd.  Since $2^f \ge 8$, we have $2^f - 1$ is divisible by at least two primes if $f$ is even and $2^f + 1$ is divisible by at least two primes if $f$ is odd.  It follows that  $G$ violates the two-prime hypothesis, so this case cannot occur.  We may assume $m$ is an odd prime. Then $G = S \langle \varphi^{\frac{f}{m}} \rangle$ and $\cd(G) = \{ 1, 2^f,  2^f-1, (2^f-1)m, 2^f+1, (2^f+1)m \}$.  Since $G$ satisfies the two-prime hypothesis, both $2^f - 1$ and $2^f + 1$ are divisible by at most two primes, as listed in Table \ref{tab:L2}.

Now the proof is complete.
\end{proof}

We will need the maximal subgroups of $\PSL_2 (q)$ for $q \geq 7$. According to Dickson's list of subgroups of $\PSL_2(q)$ (cf. \cite{Huppert}, Hauptsatz II. $8.27$), we have a list of the maximal subgroups of  $\PSL_2(q)$.  We first list the maximal subgroups when the characteristic is even.

\begin{theorem} {\rm (\cite[Theorem $2.1$]{Giudici}).}\label{maximaleven}
 If $q= 2^f \geq 8$, then the maximal subgroups of $\PSL_2(q)$ are$:$

\begin{enumerate}
   \item $C_2^f \rtimes C_{q-1}$, that is, the stabilizer of a point of the projective line,
  \item ${\rm{D}}_{2(q-1)},$
  \item ${\rm{D}}_{2(q+1)},$
  \item ${\rm PGL}(2, q_0)$, where $q = q_0^r$ for some prime $r$ and $q_0 \neq 2$.
\end{enumerate}
\end{theorem}

Here is the list of maximal subgroups when the characteristic is odd (we shall directly refer to the Atlas \cite{Atlas} when $q=7,9,11$.)

\begin{theorem} {\rm (\cite[Theorem $2.2$]{Giudici}).}\label{maximalodd}
If $q=p^f\geq 13$, where $p$ is an odd prime, then the maximal subgroups of $\PSL_2(q)$ are$:$

\begin{enumerate}[$(1)$]
  \item $C_p^f\rtimes C_{(q-1)/2}$, that is, the stabilizer of a point of the projective line,
  \item ${\rm D}_{q\pm1}$,
  \item ${\rm PGL}(2, q_0)$, for $q = q_0^2$,
  \item $\PSL_2( q_0)$, for $q = q_0^r$, where $r$ is  an odd prime,

  \item ${\rm A}_5$, for $q\equiv \pm 1$ $({\rm mod}\ 10)$, with either $q $ prime, or $q = p^2$ and $p\equiv \pm 3$ $({\rm mod}\ 10)$,

  \item ${\rm A}_4$, for $q = p \equiv \pm 3$ $({\rm mod}\ 8)$; and $q \not\equiv \pm 1$ $({\rm mod}\ 10)$,

  \item ${\rm S}_4$, for $q \equiv \pm 1$   $({\rm mod}\ 8)$ with either $q$ prime, or $q=p^2$ and $3<p\equiv \pm 3$ $({\rm mod}\ 8)$.
\end{enumerate}
\end{theorem}


\section{$G/N\cong \PSL_2(q)$}

In this section, we prove the following theorem.

\begin{theorem}\label{th:L2q} Let  $G$ be a group satisfying the two-prime hypothesis.  If $G/N$ is an almost simple group with simple socle $M/N\cong\PSL_2(q)$ with $q\ge 7,$ then $|\cd(G)| \leq 19$.
\end{theorem}

Assume the hypothesis of Theorem \ref{th:L2q}. Let $\cA_1\subseteq \Irr(N)$ be a set consisting of all $\theta\in\Irr(N)$ such that $\theta$ is $M$-invariant and let $\cA_2 \subseteq \Irr(N)$ be the set consisting of all $\theta\in\Irr(N)$ which is not $M$-invariant. For $i=1,2,$ let $\cB_i$ be the union of all $\cd (G | \theta)\setminus\cd(G/N)$ for all $\theta\in\cA_i.$ Then $\Irr(N)=\cA_1\cup \cA_2$ and $\cd(G)=\cB_1\cup \cB_2\cup \cd(G/N).$

When $q$ is odd, fix $\epsilon \in \{ \pm 1 \}$ so that $q\equiv \epsilon$ (mod $4$). Then $q-\epsilon$ is divisible by at least three primes. Notice that $G/M$ is always cyclic unless $G/N \cong \PGL_2(q) \langle \varphi^{f/2} \rangle$ with $f$ even and $q$ odd, where $G/M\cong C_2^2$. Moreover, if $q$ is odd, then $|G:M|\in\{1,2,f,2^2\}$.

\begin{lemma}\label{lem:L2MInvariant} Under the hypothesis of Theorem \ref{th:L2q}, we have $|\cB_1|\le 4$ if $q\neq 9.$
\end{lemma}

\begin{proof}
Let $\theta\in\cA_1$ and let $I=I_G(\theta)$. Then $M\leq I\le G.$ We consider the following cases.

(\textbf{1}) $I\neq G.$  Write  $\theta^M = \sum e_i\phi_i,$ where $\phi_i\in\Irr(M|\theta)$.

(a) Assume that $q=2^f\ge 8.$ As the Schur multiplier of $M/N$ is trivial and $M/N$ is perfect, $\theta$ has a unique extension $\theta_0\in\Irr(M).$ So, $\theta_0$ is $I$-invariant and since $I/M$ is cyclic, $\theta_0$ extends to $I$ and hence $\theta$ extends to $I.$ As $I/N$ has an irreducible character $\gamma$ of degree $q$ which is the extension of the Steinberg character of $M/N$ of degree $q.$
 By Gallagher's theorem, $I$ has an irreducible character $\phi\in\Irr(I|\theta)$ of degree $q\theta(1)$. Thus $\phi^G\in\Irr(G)$ and $\phi^G(1)=|G:I|\phi(1)=|G:I|q\theta(1)$. However, this violates the two-prime hypothesis since $q\in\cd(G/N)$ and $q$ is divisible by three primes.

(b) $q$ is odd and $|G/M|$ is a prime. Then $I=M$.
It follows that $\cd(I|\theta)=\{1,(q+\epsilon)/2,q,q\pm 1\}\theta(1)$ or $\{(q-\epsilon)/2,q\pm 1\}\theta(1)$. Hence $\cd(G|\theta)=\{1,(q+\epsilon)/2,q,q\pm 1\}\theta(1)|G:I|$ or $\{(q-\epsilon)/2,q\pm 1\}\theta(1)|G:I|.$ Since $\gcd(q-1,q+1)=2$ and $|G:I|>1$, we have $\theta(1)=1.$ Therefore, $\cd(G|\theta)\setminus\cd(G/N)\subseteq \{1,(q+\epsilon)/2\}|G:M|$.

(c) $q$ is odd and $|G/M|$ is not a prime. Then $G/N\cong \PGL_2(3^f)\langle \varphi\rangle$ with $f$ an odd prime and $3^f-1$ a product of two primes; or $\PGL_2(q)\langle \varphi^{f/2}\rangle $ where $f$ is even, $q$ is odd and $q+1$ is divisible by at most two primes.

Since $|G:M|$ is divisible by two primes and $(q\pm 1)\theta(1)\in\cd(M|\theta)$, we deduce that $I\neq M.$ Let $\psi_i\in\Irr(I|\theta)$ such that $q-(-1)^i$ divides $\psi_i(1)$ for $i=1,2.$ Clearly $\psi_1(1)\neq \psi_2(1)$ so $\psi_i^G(1)=|G:I|\psi_i(1),i=1,2$ are two distinct degrees of $G.$ It follows that $\theta(1)=1.$

Assume first that $G/N\cong\PGL_2(3^f)\langle \varphi\rangle.$ If $\theta$ extends to $M$ then it is extendible to $\theta_0\in\Irr(I)$ as $I/M$ is cyclic. So by Gallagher's theorem, $\gamma=\theta_0\mu\in\Irr(I|\theta)$ where $\mu$ is an extension of the Steinberg character of $M/N$ to $I/N.$ Then $\gamma^G\in\Irr(G|\theta)$ with $\gamma^G(1)=|G:I|q\theta(1)>q=3^f.$ Since $f\ge 3,$ we obtain a contradiction. Thus $\theta$ is not extendible to $M$ and hence if $\theta^M=\sum f_i\psi_i$ with $\psi_i\in\Irr(M|\theta)$, then $\{f_i\}=\{(q+1)/2,q\pm 1\}$. Notice that $q=3^f\equiv -1$ (mod $4$). It follows that $\cd(I|\theta)\subseteq \{(q+1)/2,q\pm 1\}\cup \{(q+1)/2,q\pm 1\}\cdot |I:M|$ and so \[\cd(G|\theta)\subseteq \{(q+1)/2,q\pm 1\}\cdot|G:I|\cup \{(q+1)/2,q\pm 1\}\cdot 2f.\] Since $\cd(G/N)=\{1,q-1,q,(q-1)f,(q+1)f\}$ and $G$ satisfies the two-prime hypothesis, we deduce that $\cd(G|\theta)\setminus \cd(G/N)\subseteq \{2(q-1)\}.$

Assume next that $G/N\cong\PGL_2(q)\langle \varphi^{f/2}\rangle $, where $f$ is even, $q$ is odd and $q+1$ is divisible by at most two primes.
In this case $$\cd(I|\theta)\subseteq \{1,q,(q\pm1)/2,q\pm 1\}\cup \{1,q,(q\pm1)/2,q\pm 1\}\cdot 2$$ and so \[\cd(G|\theta) \subseteq \{2,2q,q\pm1,2(q\pm 1)\}\cup \{4,4q,2(q\pm1),4(q\pm 1)\} .\] Since $\cd(G/N)=\{1,q+1,q,2(q\pm1)\}$ and $G$ satisfies the two-prime hypothesis, we deduce that $\cd(G|\theta)\setminus \cd(G/N)\subseteq \{2,4,2q,4q\}.$ The last two degrees cannot occur at the same time as $q$ is divisible by at least two primes.

(\textbf{2}) $I= G.$

(a) $\theta$ extends to $\theta_0\in\Irr(G).$ Using Gallagher's theorem and the fact that $\cd(G/N)$ contains a degree divisible by three primes, we have $\theta(1)=1$ and $\cd(G|\theta)\subseteq \cd(G/N)$.

 (b) $\theta$ extends to $\theta_0\in\Irr(M)$ but $\theta$ is not extendible to $G.$ It follows that $G/M$ is noncyclic.  Then $G/N\cong\PGL_2(q)\langle\varphi^{f/2}\rangle$ with $q$ odd, $f$ even. Let $T/N\cong \PGL_2(q).$ Then $\theta$ extends to $\theta_0\in\Irr(T).$ So $\cd(T|\theta)=\{1,q,q\pm 1\}\theta(1).$
Since $\cd(G|\theta)$ contains a degree divisible by $(q-1)\theta(1)$, where $q-1$ is divisible by three primes and $2(q-1)\in\cd(G/N)$, we deduce that  $\theta(1)=1$ or $2.$ If $\theta(1)=1,$ then $\cd(G|\theta)\setminus\cd(G/N)\subseteq \{2,2q\}$; and if $\theta(1)=2$, then $\cd(G|\theta)\setminus\cd(G/N)\subseteq \{2,4,2q,4q\}$. The last two degrees cannot occur at the same time.

(c) $\theta$ is not extendible to $M.$ Then $q$ is odd. Write $\theta^M=\sum f_i\psi_i.$ Then $\cd(M|\theta)=\{q\pm \epsilon,(q-\epsilon)/2\}\cdot \theta(1)$. In particular, $\cd(G|\theta)$ contains a degree, say $d$, divisible by $(q- \epsilon)\theta(1).$ We consider the following cases.

(i) $G/N\cong\PSL_2(q)$ or $\PGL_2(q).$ Then $q-\epsilon\in\cd(G/N).$ It follows that $\theta(1)=1$ as $q-\epsilon$ is divisible by at least three primes.
Since $|G:M|\le 2,$ we have $$\cd(G|\theta)\subseteq \{q\pm \epsilon,(q-\epsilon)/2\}\cup \{q\pm \epsilon,(q-\epsilon)/2\}\cdot |G:M|.$$ As $\{q\pm \epsilon\}\subseteq \cd(G/N)$ we deduce that $\cd(G|\theta)\setminus \cd(G/N)\subseteq \{2(q+\epsilon),(q-\epsilon)/2\}.$

(ii) $f$ is even and $|G:M|=2.$ In this case, $\epsilon=1$ and $\{q+\epsilon,2(q\pm \epsilon)\}\subseteq \cd(G/N).$ Since $d\in\cd(G|\theta)$ is divisible by $(q-\epsilon)\theta(1)$, we see that  $\theta(1)=1$ or $2.$

 If $\theta(1)=2,$ then $$\cd(G|\theta)\subseteq \{2(q\pm \epsilon),q-\epsilon\}\cup \{4(q\pm \epsilon),2(q-\epsilon)\}.$$ Hence $\cd(G|\theta)\setminus\cd(G/N)=\emptyset.$ Notice that $2(q+\epsilon)$ is divisible by at least three primes.

If $\theta(1)=1$ then $$\cd(G|\theta)\subseteq \{q\pm \epsilon,(q-\epsilon)/2\}\cup \{2(q\pm \epsilon),q-\epsilon\}.$$ Hence  $\cd(G|\theta)\setminus\cd(G/N)\subseteq \{(q-\epsilon)/2\}$.

 (iii) $G/N\cong\PSL_2(3^f)\langle\varphi\rangle$ or $\PGL_2(3^f)\langle \varphi\rangle$ with $f$ an odd prime. Then $\epsilon=-1$ and $\{f(3^f\pm\epsilon)\}\subseteq \cd(G/N).$ Since $d\in\cd(G|\theta)$, $\theta(1)=1$ or $f$.

 Assume $\theta(1)=f.$ As $|G:M|$ divides $2f$, we deduce that $$\cd(G|\theta)\subseteq \{f(3^f\pm \epsilon),f(3^f-\epsilon)/2\}\cdot  \{1,2,f,2f\}.$$ Hence $\cd(G|\theta)\setminus\cd(G/N)=\emptyset.$

 Assume $\theta(1)=1.$ As $|G:M|$ divides $2f$, we deduce that $$\cd(G|\theta)\subseteq \{3^f\pm \epsilon,(3^f-\epsilon)/2\}\cdot  \{1,2,f,2f\}.$$ Hence if $G/N\cong\PGL_2(3^f)\langle\varphi\rangle,$ then $\cd(G|\theta)\setminus\cd(G/N)\subseteq \{(3^f-\epsilon)/2,2(3^f+\epsilon)\}$ and if $G/N\cong\PSL_2(3^f)\langle\varphi\rangle$, then $\cd(G|\theta)\setminus\cd(G/N)\subseteq \{(3^f-\epsilon)/2,3^f+\epsilon\}.$

(iv) $G/N\cong\PGL_2(q)\langle\varphi^{f/2}\rangle$, $q$ odd, $f$ even. Then $\epsilon=1$ and $\{q+1,2(q\pm 1)\}\subseteq \cd(G/N).$
In this case, we obtain that $\theta(1)\in\{1,2\}$ and \[\cd(G|\theta)\subseteq \{(q\pm 1)\theta(1),(q-1)\theta(1)/2\}\cdot\{1,2,4\}.\]
 In both cases, we have that $\cd(G|\theta)\setminus\cd(G/N)\subseteq \{(q-1)/2\}$.

 We now find the upper bound for $|\cB_1|$ for each possibility of $G/N$.
Assume first that $G/N\cong\PGL_2(3^f)\langle\varphi\rangle$ with $f$ an odd prime. It follows from Cases 1(c) and 2c(iii) that $\cB_1\subseteq \{2(3^f-1),(3^f+1)/2\}$ and so $|\cB_1|\le 2.$

Assume next that $G/N\cong\PGL_2(q)\langle\varphi^{f/2}\rangle$ with $q$ odd and $f$ even. It follows from Cases 1(c), 2(b) and 2c(iv) that $\cB_1\subseteq \{2,4,2q,4q,(q-\epsilon)/2\}$ and so $|\cB_1|\le 4.$

For the remaining cases, we can check that $|\cB_1|\le 4.$ The proof is now complete.
\end{proof}


\begin{lemma}\label{lem:L2(8)} Theorem \ref{th:L2q} holds if $q=8.$
\end{lemma}

\begin{proof} We have that $G/N\cong\PSL_2(8)$ or $\PSL_2(8)\cdot 3$ and $\cd(G/N)=\{1,7,8,9\}$ or $\{1,7,8,21,27\}$, respectively. Moreover, the Schur multiplier of $\PSL_2(8)$ is trivial.

Let $\theta\in\cA_2$ and let $I=I_G(\theta).$ Then $M\nleq I$ and since $G/M$ is trivial or cyclic of prime order, we have $G=MI$ so $|G:I|=|M:I\cap I|.$ Let $M\cap I\leq T\leq M$ such that $T/N$ is maximal in $M/N.$

If $|M:T|$ is divisible by three primes, then every degree in $\cd(G|\theta)$ is divisible by $|M:T|$ so $|\cd(G|\theta)|=1.$ Notice that $M/N$ has two such maximal indices which are $2^2\cdot 7$ and $2^2\cdot 3^2$.  If $|M:T|$ is divisible by at most two primes, then $T/N\cong 2^3:7$, a Frobenius group of order $56$.  If $M\cap I\neq T,$ then every degree in $\cd(G|\theta)$ is divisible by $2\cdot 9=18$ or $7\cdot 9=63$ so $|\cd(G|\theta)|\le 2$ in this case. Otherwise $T=M\cap I$. If $\theta$ extends to $T,$ then $\cd(M|\theta)=\{1,7\}\cdot 9\theta(1)$ so $\theta(1)=1$ and $\cd(G|\theta)\subseteq \{9,9\cdot 7\}\cup \{9\cdot 2, 2\cdot 7\cdot 9\}.$ If $\theta$ is not extendible to $T,$ Then $\theta$ is not extendible to $N_1$ where $N_1/N\cong 2^3.$ So, every degree in $\cd(T|\theta)$ is even and thus every degree in $\cd(G|\theta)$ is divisible by $2\cdot 9.$ Therefore, every degree in $\cB_2$ is divisible by $2^2\cdot 7,2\cdot 3^2,7\cdot 3^2$ (one degree each)  or  is the degree $9$. Hence $|\cB_2|\le 4.$

Since $|\cB_1|\le 4$ by Lemma \ref{lem:L2MInvariant} and $|\cd(G/N)|\le 5,$ we have $|\cd(G)|\le 5+4+4=13.$
\end{proof}

\begin{lemma}\label{lem:L2(9)} Theorem \ref{th:L2q} holds if $q=9.$
\end{lemma}

\begin{proof}
We have that $G/N\cong\Alt_6,\Sym_6=\Alt_6\cdot 2_1, \PGL_2(9)=\Alt_6\cdot 2_2,\textrm{M}_{10}=\Alt_6\cdot 2_3$ or $\Alt_6\cdot 2^2.$

Let $\theta\in\cA_2$ and let $I=I_G(\theta).$ Then $I\leq IM\leq G$ and $M\nleq I.$ We see that $|G:I|$ is divisible by $|IM:I|=|M:I\cap M|.$ Let $M\cap I\le T\leq M$ such that $T/N$ is maximal in $M/N.$ From \cite{Atlas}, we have that $T/N\cong \Alt_5,3^2:4$ or $\Sym_4$ with index $2\cdot 3, 2\cdot 5$ and $3\cdot 5,$ respectively in $M/N$. Observe first that $|G:M|=1,2$ or $4.$

If $T\neq M\cap I,$ then every degree in $\cd(G|\theta)$ is divisible by one of the numbers $\{5,2\cdot 3,2\cdot 5\}\cdot 2\cdot 3$, $\{2,3^2\}\cdot 2\cdot 5$ or $\{2,3,2^2\}\cdot 3\cdot 5.$ So every degree in $\cd(G|\theta)$ is divisible by $2\cdot 3\cdot 5, 3^2\cdot 5,2^2\cdot 3^2$ or $2^2\cdot 5.$

Assume that $T=M\cap I.$ Suppose further that $\theta$ extends to $T.$ If $T/N\not\cong \Sym_4,$ then $\cd(G|\theta)$ contains characters of degree divisible by $8$ but  not a power of $2$, which is impossible as $8$ or $16$ is a degree of $G/N$. Assume now that $T/N\cong \Sym_4.$ Then $\cd(M|\theta)=\{1,2,3\}\cdot3\cdot 5\theta(1)$. It follows that $\theta(1)=1$ and $\cd(G|\theta)\subseteq \{3\cdot 5a, 2\cdot 3\cdot 5a,3^2\cdot 5a |a=1,2,4\}$ since $|G:M|$ divides $4$. By the two-prime hypotheses,  $\cd(G|\theta)$ contains possibly $3\cdot 5,3^2\cdot 5$ and a degree divisible by $2\cdot 3\cdot 5.$

Assume that $\theta$ does not extend to $T.$ Then every degree in $\cd(T|\theta)$ is divisible by 2,3 or 2, respectively. So, every degree in $\cd(G|\theta)$ is divisible by $2^2\cdot 3, 2\cdot 3\cdot 5$ or $2\cdot 3\cdot 5,$ respectively.

Therefore, if $\theta\in\cA_2,$ then every degree in $\cB_2$ is divisible by $2\cdot 3\cdot 5,2^2\cdot 3,2^2\cdot 5,3^2\cdot 5$ or is the degree $3\cdot 5$ and thus $|\cB_2|\le 5.$

\medskip
Let $\theta\in\cA_1$ and let $I=I_G(\theta).$ Then $M\leq I.$

Assume that $\theta$ extends to $M.$ Then $\cd(M|\theta)=\{1,5,2^3,3^2,2\cdot 5\}\theta(1)$. Since $8$ or $16$ is in $\cd(G/N),$ we deduce that $\theta(1)=1 $ or $2.$ So, $\cd(G|\theta)\subseteq  \cup_{1\leq a\mid 8}\{1,5,2^3,3^2,2\cdot 5\}\cdot a$. Moreover, $\{9,10\}\subseteq \cd(G/N)$, we deduce that $\cd(G|\theta)\setminus\cd(G/N)\subseteq \{2,4,5,18,36,20\}$. Notice that $18$ and $36$ cannot occur at the same time.

Assume next that $\theta$ is $M$-invariant but not extendible to $M.$ Then $\cd(M|\theta)= \{4,8,10\}\theta(1), \{3,6,9,15\}\theta(1)$ or $ \{6,12\}\theta(1).$ Since $4\mid 8 $ and $6\mid 12$, we deduce that $\theta(1)=1$ in the first or the last case and $\theta(1)=1$ or $2$ in the second case. Since $8$ or $16$ is a degree of $G/N,$  $8,16,24\not\in\cd(G|\theta)\setminus \cd(G/N).$ It follows that $\cd(G|\theta)\subseteq \{3,4,6,12,15,20,18,30,36,40,60\}$.  Recall that $\{9,10\}\subseteq \cd(G/N)$.

Thus $\cB_1\subseteq \{2,3,4,5,6,12,15,18,20,30,36,40,60\}$. Since $30$ and $60$ are divisible by $2\cdot 3\cdot 5$,  $12$ and $36$ are divisible by $2^2\cdot 3$; $20$ and $40$ are divisible by $2^2\cdot 5$  and $15$ is possibly in $\cB_2,$ we see that $|\cB_1\cup \cB_2|\le 10+5-4=11.$  Therefore, since $|\cd(G/N)|\le 5,$ we deduce that $|\cd(G)|\le |\cd(G/N)|+|\cB_1\cup \cB_2|\le 5+11=16.$
\end{proof}

\begin{lemma}\label{lem:L2even} Assume the hypothesis of Theorem \ref{th:L2q} and suppose that $q=2^f>8.$ Then $|\cB_2|\le 5$.
\end{lemma}

\begin{proof} Let $\theta\in\cA_2$ and let $I=I_G(\theta).$ Then $M\nleq I$ so $I\leq MI\le G$ and thus $|G:I|$ is divisible by $|M:T|$ where $M\cap I\le T\leq M$ such that $T/N$ is a maximal subgroup of $M/N\cong \PSL_2(2^f)$ with $f\ge 4.$

If $T/N\cong \textrm{D}_{2(q\pm 1)}$ or $\PSL_2(q_0)$ with $q=q_0^r,q_0\neq 2,r$ a prime, then $|M:T|$ is divisible by at least three primes, so $|\cd(G|\theta)|=1.$

So, assume that $T/N\cong 2^f:(2^f-1)$ which is a Frobenius group with kernel $2^f.$ Observe that $M\cap I/N$ is a cyclic subgroup of a cyclic group of order $2^f-1$, an elementary abelian $2$-group of order $2^e$ with $1\le e\le f$ or a Frobenius group of order $2^e:d$ where $d$ divides $\gcd(2^e-1,2^f-1)$ and $1\le e\le f.$ Moreover, since $f\ge 4,$ $q^2-1=4^f-1$ is divisible by at least three primes.

We claim that every degree in $\cd(G|\theta)$ is divisible by $q^2-1$ or $4(q+1)$.
In both cases, $|\cd(G|\theta)|\le 1.$
Let $L/N\cong 2^f.$
Assume that $L\nleq M\cap I.$  Then $|T:M\cap I|$ is even.

If $|T:M\cap I|$ is divisible by $4$ then $|\cd(G|\theta)|=1$ and every degree in $\cd(G|\theta)$ is divisible by $4(q+1).$

Assume that $|T:M\cap I|_2=2.$ Then every degree in $\cd(G|\theta)$ is divisible by $2(q+1)$. If $\theta$ is not extendible to the Sylow $2$-subgroup of $M\cap I/N,$ then every degree in $\cd(M\cap I|\theta)$ is even so every degree in $\cd(G|\theta)$ is divisible by $4(q+1).$ Therefore, we can assume that $\theta$ extends to the Sylow $2$-subgroup of $M\cap I/N$ and thus it extends to $M\cap I.$ Hence $\cd(G|\theta)$ contains a degree divisible by $q^2-1.$

Assume $L\leq M\cap I$ where $L/N\cong 2^f.$ We can assume that $\theta$ is not extendible to $L$ as otherwise $q^2-1$ divides some degree in $\cd(G|\theta).$ In this case, $M\cap I/N=2^f:d$ where $d\mid 2^f-1.$

By Theorem 2.7 of \cite{Isaacs73}, $\theta$ determines a unique subgroup $A$ with $N \le A \le L$ satisfying $\theta$ extends to $A$ and every character in $\Irr(A| \theta)$ is fully ramified with respect to $L/A$.  This implies that $|L:A| = 2^{2a}$ for some integer $a$.  We know that $2^a$ will divide every degree in $\cd(L|\theta)$.  If $a > 1$, then $4(q+1)\theta(1)$ divides all degrees in $\cd(G|\theta)$.  Thus, we may assume that $a = 1$.  Since $f \ge 4$ and $|A:L| = 2^{f - 2a}$, we conclude that $L < A$.  Let $H/N$ be a Frobenius complement for $T/N$.  Since $H$ stabilizes $\theta$ and $A$ is determined by $\theta$, we see that $H$ will normalize $A$.  Because $T/N$ is a Frobenius group, $H/L$ will have a regular orbit on $\Irr(A/N)$.  Applying Gallagher's theorem, this yields a regular orbit for $H/N$ on $\Irr(A|\theta)$, and thus, a regular orbit for $H/N$ on $\Irr(L| \theta)$.  In particular, there exists a character $\hat\theta \in \Irr(L| \theta)$ whose stabilizer in $I$ is $L$.  It follows that $\hat\theta^I \in \Irr(I | \theta)$, and so, $\hat\theta^G \in \Irr(G| \theta)$.  We now have $(q + 1)(q - 1) 2 \theta (1) = |G:L| 2 \theta (1) = \hat\theta^G (1) \in \cd(G| \theta)$, and this proves that $(q^2-1)\theta(1)$ divides all degrees in $\cd(G|\theta)$.

In summary, we deduce that $|\cB_2|\le 3+2=5.$
\end{proof}

\begin{lemma}\label{lem:L2odd} Assume the hypothesis of Theorem \ref{th:L2q} and suppose that $q=p^f\ge 13$ with $p>2.$ Then $|\cB_2|\le 9$.
\end{lemma}

\begin{proof} Let $\theta\in\cA_2$ and let $I=I_G(\theta).$ Then $M\nleq I$ so $I\leq MI\le G$ and thus $|G:I|$ is divisible by $|M:T|$ where $J:=M\cap I\le T\leq M$ such that $T/N$ is a maximal subgroup of $M/N\cong \PSL_2(p^f)$ with $f\ge 1.$

\medskip
(1) $T/N\cong C_p^f:C_{(p^f-1)/2}.$ We claim that $p(q+1)$ divides every degree in $\cd(G|\theta)$ and since $p(q+1) $ is divisible by three primes, we have $|\cd(G|\theta)|\leq 1.$

We have $|M:T| = q + 1$.   Recall that either $\omega (q - 1) \ge 3$ or $\omega (q + 1) \ge 3$ and both $q-1$ and $q+1$ divide $(q-1)(q+1)/2$ and $q\pm 1,f(q\pm 1)$ or $2(q\pm 1)$ are in $\cd(G/N)$.  Thus, by the two-prime hypothesis, we cannot have $(q-1)(q+1)/2$ dividing some degree in $\cd(G| \theta)$.  Let $N_1/N \leq T/N$ with $N_1/N \cong  C_p^f$.  If $J \cap N_1< N_1$, then $p$ divides every degree in $\cd(G| \theta)$, and the claim follows.  Thus, we may assume $N_1 \le J$.  If $J = N_1$, then $|T:J| = (q-1)/2$ and $(q-1)(q+1)/2$ divides the degrees in $\cd(G | \theta)$, a contradiction.  Thus  $N_1 < J$.  Notice that $J/N$ is a Frobenius group with Frobenius kernel $N_1/N$.  Let $H/N$ be a Frobenius complement for $J/N$.  If $\theta$ does not extend to $N_1$, then $p$ will divide every degree in $\cd(G| \theta)$.  We assume that $\theta$ extends to $N_1$.  By Glauberman's lemma \cite[Lemma 13.8]{Isaacs} and \cite[Corollary 13.9]{Isaacs}, there will be a unique extension $\hat{\theta} \in \Irr(N_1| \theta)$ that is $H$-invariant.  It follows that $\hat\theta$ is $J$-invariant, and since $J/N_1$ is cyclic, $\hat\theta$ extends to $J$.  This implies that $\theta$ extends to $J$, and by Gallagher's theorem we have $|J:N_1| \in \cd(J | \theta)$.  This implies that $|M:N_1| = (q-1)(q+1)/2$ divides some degree in $\cd(G \mid \theta)$, a contradiction.

\medskip
(2) $T/N \cong \PGL_2(q_0)$, where $q = q_0^2$. We claim that   every degree in $\cd(G|\theta)$ is divisible by $p(q+1)$ or $p^2(q+1)/2.$ Clearly, $q_0>3$.
 Note that this implies $f$ is even and $|M:T|  =  {q_0 (q+1)}/2$.  Observe that $\omega (|M:T|) \geq 2$.

 Observe next that $(q-1)(q+1)q_0/4$ is divisible by either $q-1$ or $q+1$ and thus  $\cd(G|\theta)$ has no degree divisible by $q_0(q^2-1)/4$ as $\cd(G/N)$ has degrees $q\pm 1,2(q\pm 1)$ or $f(q\pm 1).$

Let $N_1/N\lhd T/N$ with $N_1/N\cong \PSL_2(q_0)$.
If $N_1\leq J $, then $\{ q_0 - 1, q_0 + 1 \} \cdot \theta (1) \subseteq \cd(N_1| \theta)$. Recall that $N_1\unlhd J\unlhd I.$ Now if $\cd(I|\theta)$ has two distinct degrees, one divisible by $q_0-1$ and another divisible by $q_0+1$, then  $\cd(G| \theta)$ has two distinct degrees divisible by $q_0(q-1)$, violating the two-prime hypothesis. Thus $\cd(I|\theta)$ has a degree divisible by $(q_0-1)(q_0+1)/2$ which is $\lcm(q_0-1,q_0+1).$ Hence $\cd(G|\theta)$ has a degree divisible by $(q-1)(q+1)q_0/4$ which contradicts the observation above.

Assume $N_1\nleq J.$ Then $p$ divides $|JN_1:J|=|N_1:N_1\cap J|$ or $|JN_1:J| = q_0 + 1$, which implies that every degree in $\cd(G | \theta)$ is divisible by either $q_0(q+1) (q_0 + 1)/2$ or $q_0(q+1) p/2$. In particular, every degree in $\cd(G|\theta)$ is divisible by $p(q+1)$ or $p^2(q+1)/2.$

 \medskip
 (3) $T/N\cong \PSL_2(q_0)$, with $q=q_0^r$ where $r$ is an odd prime.

 We see $|M:T|=q_0^{r-1}(q^2-1)/(q_0^2-1)$. Since $r\ge 3$, $|M:T|$ is divisible by at least three distinct primes and thus $|\cd(G|\theta)|\leq 1.$ Notice that  $q_0^{r-1}(q^2-1)/(q_0^2-1)$ and $p^2(q+1)/2$ have at least three prime divisors in common.

 \medskip
 (4) $T/N\cong \textrm{D}_{q-\delta}$ with $\delta=\pm 1\ge 1$. We have $|M:T|=q(q+\delta)/2.$

 If $q(q+\delta)/2$ is divisible by at least three primes, then $|\cd(G|\theta)|=1.$ So, assume $q(q+\delta)/2$ is divisible by at most two primes. Then $q=p$ is prime, $\delta=\epsilon$ and $(q+\epsilon)/2$ is prime, where $q\equiv \epsilon$ (mod $4$). It follows that $G/N\cong \PSL_2(q)$ or $\PGL_2(q)$ and so $I/N\leq H/N$ where $H/N\cong \textrm{D}_{2n}$ with $n=(q-\epsilon)/2$ or $q-\epsilon,$ respectively. Since $q\ge 13,$  $H/N$ is nonabelian and $\cd(H/N)=\{1,2\}.$

 If $\theta$ extends to $H,$ then $\cd(G|\theta)=\{q(q+\epsilon)/2,q(q+\epsilon)\}\cdot\theta(1)$ which implies that $\theta(1)=1.$

  If $\theta$ is $H$-invariant but not extendible to $H,$ then every degree in $\cd(G|\theta)$ is divisible by $q(q+\epsilon)$.

  Finally, assume that $\theta$ is not $H$-invariant. Then  $I/N$ is a proper subgroup of $H/N.$ It follows that $I/N$ is cyclic whose order divides $n$ or $I/N\cong \textrm{D}_{2d}$ with $d\mid n.$ If $|H:I|$ is even, then every degree in $\cd(G|\theta)$ is divisible by $q(q+\epsilon)$. So, assume $|H:I|$ is odd, which implies that either $I/N$ is a nonabelian dihedral group of order $2d$ with $n\neq d\mid n$ or $I/N\cong C_2$ or $C_2^2.$ In all cases, $|G:I|$ is divisible by at least three primes, so $|\cd(I|\theta)|=1.$ If $I/N$ is a nonabelian dihedral group, then $\theta$ is not extendible to $I$ and so $\cd(G|\theta)=\{q(q+\epsilon)\}\theta(1)$ as $\cd(I|\theta)=\{2\theta(1)\}.$ For the remaining two cases, we see that the degree in $\cd(G|\theta)$ is divisible by $q(q-1)(q+1)/8.$

  \medskip
 (5) $T/N\cong \Alt_5$ where $q\equiv \pm 1$ $({\rm mod}\ 10)$, with either $q $ prime, or $q = p^2$ and $p\equiv \pm 3$ $({\rm mod}\ 10)$. Then $|M:T|=q(q^2-1)/120$ and $\gcd(q,10)=1.$ Observe that if $p$ is any odd prime, then $p^2\equiv 1$ (mod $8$) so $16\mid p^4-1.$

 Since $q\ge 13,$ we can assume $q\ge 19.$  It follows that $q\neq 3,9$ so $3\nmid q$ and thus $120\mid q^2-1.$ Write $q^2-1=120 t$ for some integer $t>1.$ Hence $|M:T|=qt>q.$ If $qt$ is divisible by three primes, then $|\cd(G|\theta)|\le 1.$ So, assume now that $qt$ is divisible by at most two primes which implies that $q=p$ and $t$ is a prime. Hence $G/N\cong \PSL_2(p)$ or $\PGL_2(p).$ From \cite[Theorem 3.5]{Giudici}
, we can see that the second case cannot occur.  Hence $G/N\cong \PSL_2(p)$ and $I/N\leq T/N\cong \Alt_5.$

If $\theta$ is $T$-invariant, then either $\cd(G|\theta)=\{qt,3qt,4qt,5qt\}\theta(1)$ or $\{2qt,4qt,6qt\}\theta(1).$ Clearly, the second case cannot hold and if the first case holds, then $\theta(1)=1.$

 If $\theta$ is not $T$-invariant, then $|T:I|$ is divisible by $5,6$ or $10.$ So $|G:I|$ is divisible by $5qt,6qt$ or $10qt$ which are all divisible by at least three primes. So $|\cd(G|\theta)|\le 	1$ and every degree in $\cd(G|\theta)$ is divisible by $5qt$ or $6qt.$

  \medskip
 (6) $T/N\cong \Alt_4$ where $q = p \equiv \pm 3$ $({\rm mod}\ 8)$ and $q\not\equiv \pm 1$ (mod $10$). Then $|M:T|=q(q^2-1)/24$ and since $q>7$ is odd, we have $24\mid q^2-1$ so $(q^2-1)/24=r$ for some integer $r>1.$ Notice that $G/N\cong \PSL_2(q)$ or $G/N\cong\PGL_2(q).$

 If $qr$ is divisible by at least three primes, then $|\cd(G|\theta)|\le 1.$ So, assume $r$ is a prime.

Assume first that $G/N\cong\PGL_2(q)$ and let $I/N\leq H/N$ where $H/N\cong \Sym_4.$ In this case, $H/N\cap M/N\cong T/N.$ If $\theta$ is $H$-invariant, then $\cd(G|\theta)=\{qr,2qr,3qr\}\theta(1)$ or $\{2qr,4qr\}\theta(1).$ The latter cannot occur and if the former case occurs, then $\theta(1)=1.$ If $\theta$ is not $H$-invariant, then $|H:I|$ is divisible by $2,3$ or $4$ so $|G:I|$ is divisible by $2qr$ or $3qr$. Hence $|\cd(G|\theta)|\le 1$.

Assume next that  $G/N\cong\PSL_2(q)$.  If $\theta$ is $T$-invariant, then $\cd(G|\theta)=\{qr,3qr\}\theta(1)$ or $\{2qr\}\theta(1).$  If the former case occurs, then $\theta(1)=1.$ If $\theta$ is not $T$-invariant, then $|T:I|$ is divisible by $3$ or $4$ so $|G:I|$ is divisible by $3qr$ or $4qr$. Hence $|\cd(G|\theta)|\le 1$.

  \medskip
 (7) $T/N\cong \Sym_4$ where $q=p \equiv \pm 1$   $({\rm mod}\ 8)$. Then $|M:T|=q(q^2-1)/48.$ Observe that $16\mid q^2-1.$  If $q=3^f>9,$ then $f\ge 4$ so
 $q(q^2-1)/48$ is divisible by at least three primes and thus $|\cd(G|\theta)|\le 1.$ If $3\nmid q,$ then $q^2-1$ is divisible by $3$ and hence $q^2-1=48s$ for some integer $s>1.$ (Assume $q>7$). If $|M:T|=qs$ is divisible by  at least three primes, then $|\cd(G|\theta)|\le 1.$ So, assume $qs$ is divisible by two primes which implies that both $q$ and $s$ are primes. From \cite[Theorem 3.5]{Giudici}, we have $G/N=M/N\cong\PSL_2(p)$ and hence $I/N\leq T/N\cong \Sym_4$.

 If $\theta$ is $T$-invariant, then $\cd(G|\theta)=\{qs,2qs,3qs\}\theta(1)$ or $\{2qs,4qs\}\theta(1)$. Obviously, the second case cannot occur and if the first case holds, then $\theta(1)=1.$ If $\theta$ is not $T$-invariant, then $|G:I|$ is divisible by $2qs,3qs$ or $4qs$ and thus $|\cd(G|\theta)|\leq 1$ and the degree in $\cd(G|\theta)$ is divisible by either $2qs$ or $3qs$.

In summary, the first three cases contribute at most three degrees in $\cB_2.$ Case (4) contributes one degree unless $q=p$ is prime and $(q+\epsilon)/2$ is prime; in this case it contributes at most three degrees $q(q+\epsilon)/2,q(q+\epsilon)$ or a degree divisible by $q(q^2-1)/8.$

For the last three cases, we see that cases $(5)$ and $(6)$ cannot occur simultaneously; the similar observation holds for cases $(6)$ and $(7)$. Moreover, if cases $(5)$ and $(7)$ occur at the same time, then both $t$ and $s$ cannot be primes simultaneously. Therefore, the maximum contribution of the last three cases to $\cB_2$ is $5$ if $q=p$ and $2$ if $q>p$.

 Assume first that $q=p.$ Then   cases (2) and $(3)$ do not occur, so the first four cases contribute at most 4 degrees to $\cB_2$ and the last four cases contribute at most $5$ degrees, so $|\cB_2|\le 4+5=9.$

 Assume that $q$ is not a prime. Then the first four cases contribute at most $6$ degrees to $\cB_2$ and the last three cases contribute at most 1 degree (in case (5) only), hence $|\cB_2|\le 6+1=7.$ Therefore $|\cB_2|\le 9$ in all cases.
\end{proof}

\begin{lemma}\label{lem:L2(7)} Theorem \ref{th:L2q} holds for $q=7$.
\end{lemma}

\begin{proof} From \cite{GAP}, we see that $G/N\cong \PSL_2(7)$ or $\PGL_2(7)$. Moreover, the maximal subgroups of $\PSL_2(7)$ are isomorphic to either $C_7:C_3$ or $\Sym_4$ and maximal subgroups of $\PGL_2(7)$ which does not contain $\PSL_2(7)$ are isomorphic to either $(C_7:C_3):C_2,\textrm{D}_{16}$ or $\textrm{D}_{12}.$

(a) Assume first that $G/N\cong\PGL_2(7).$

From Lemma \ref{lem:L2MInvariant}, we have that $|\cB_1|\le 6.$ We next claim that $|\cB_2|\le 3$ and since $|\cd(G/N)|=4,$ we obtain that $|\cd(G)|\le |\cB_1|+|\cB_2|+|\cd(G/N)|\le 6+3+4=13<19.$

We now show that $|\cB_2|\le 3.$  Let $\theta\in\cA_2$ and let $I=I_G(\theta).$ Then $M/N\nleq I/N$ so $I/N\neq G/N$. Let $H/N$ be a maximal subgroup of $G/N$ containing  $I/N$.

If $H/N\cong (C_7:C_3):C_2$, then $|G:H|=8$ so every degree in $\cd(G|\theta)$ is divisible by $8$ and since $8\in\cd(G/N),$ this contributes nothing to $\cB_2$.  If $H/N\cong \textrm{D}_{12}$, then $|G:H|=2^2\cdot 7$ so $|\cd(G|\theta)|\le 1.$

Assume that $H/N\cong \textrm{D}_{16}.$ Then $|G:H|=3\cdot 7.$ Since $H/N$ is nonabelian, we can use the argument as in Case (4) of Lemma \ref{lem:L2odd} to see that $\cd(G|\theta)=\{7\cdot 3,7\cdot 6\}$ or $|\cd(G|\theta)|=1$ and the unique degree in $\cd(G|\theta)$ is divisible by  $7\cdot 2\cdot 3.$ Therefore, $|\cB_2|\le 3.$

\smallskip
(b) Assume $G/N\cong\PSL_2(7).$ We have $\cd(G/N)=\{1,3,6,7,8\}$.

Let $\theta\in\cA_1.$ Then $\theta$ is $G$-invariant and thus by using \cite{GAP}, $\cd(G|\theta)=\{1,3,6,7,8\}\cdot \theta(1)$ or $\{4,6,8\}\cdot \theta(1).$ Since $8\in\cd(G),$ we deduce that $\theta(1)=1.$ Hence $\cB_1$ contains possibly the degree $4$  and thus $|\cB_1|\le 1.$

 Let $\theta\in\cA_2$ and let $I=I_G(\theta).$ Then $M/N\nleq I/N$ so $I/N\neq G/N$. Let $H/N$ be a maximal subgroup of $G/N$ containing  $I/N$.

 As in the previous case, if $H/N\cong C_7:C_3$, then every degree in $\cd(G|\theta)$ is divisible by $8.$ So, assume $H/N\cong\Sym_4.$ Then $|G:H|=7.$

Assume that $\theta$ is $H$-invariant. Then $\cd(G|\theta)=\{7,7\cdot 2,7\cdot 3\}\theta(1)$ or $\{2\cdot 7,4\cdot 7\}\theta(1)$. If the latter case holds, then $\theta(1)=1$. If the first case holds, then $\theta(1)=1$ or prime. Since $\theta(1)\in\cd(N)$ and $N$ is solvable, $\theta(1)$ can take at most two distinct prime values.

Assume next that $\theta$ is not $H$-invariant. Let $K/N$ be a maximal subgroup of $H/N$ containing $I/N.$
Then $K/N$ is isomorphic to $\Sym_3$, $\textrm{D}_8$, or $\Alt_4$ with indices $|H:K|$ equal to $2^2$, $3$, or $2$, respectively.  Suppose $I = K$.  If $K/N \cong \Sym_3$, then $\theta$ extends to $I$, so $\cd(G \mid \theta) = \{ 1, 2 \} \cdot 2^2 \cdot 7 \cdot \theta(1)$, a contradiction since $2^3 \in \cd(G/N)$.  If $K/N \cong \textrm{D}_8$, then $\cd(G | \theta)$ is either $\{ 1, 2 \} \cdot 3 \cdot 7 \theta(1)$ or $\{ 2 \cdot 3 \cdot 7 \theta (1) \}$. In the first case, we have $\theta (1) = 1$.  If $K/N \cong \Alt_4$, then either $\cd(G | \theta) = \{ 1, 3 \} \cdot 2 \cdot 7 \theta (1)$ and in this case $\theta (1) = 1$, or $\cd(G| \theta) = \{2^2 \cdot 7 \theta(1) \}$.  Finally, if $I \lneq K$, we see that either $2 \cdot 3$ or $2^3$ divides $|H:I|$.  Recall that $2^3 \in \cd(G)$.  If $2^3$ divides $|H:I|$, then $2^3 \cdot 7$ divides every degree in $\cd(G| \theta)$ which contradicts the two-prime hypothesis.  So we have $2 \cdot 3\mid |H:I|$, and so, $2 \cdot 3 \cdot 7 \theta (1)$ divides every degree in $\cd(G | \theta)$.

Since $\{7,8\}\subseteq \cd(G/N),$ we see that $|\cB_2|\le 11.$

Therefore, $|\cd(G)|\leq |\cB_1|+|\cB_2|+|\cd(G/N)|\le 1+11+6=18<19.$
\end{proof}

\begin{lemma}\label{lem:L2(11)} Theorem \ref{th:L2q} holds for $q=11$.
\end{lemma}

\begin{proof} We have that either $G/N\cong\PGL_2(11)$ or $\PSL_2(11).$

(a) Assume that $G/N\cong\PGL_2(11).$ We have $\cd(G/N)=\{1,10,11,12\}$.

From Lemma \ref{lem:L2MInvariant}, we have that $|\cB_1|\le 6.$ We next claim that $|\cB_2|\le 4$ and since $|\cd(G/N)|=4,$ we obtain that $|\cd(G)|\le |\cB_1|+|\cB_2|+|\cd(G/N)|\le 6+4+4=14<19.$

Let $\theta\in\cA_2$ and let $I=I_G(\theta).$ Then $I\neq G$ and let $H/N$ be a maximal subgroup of $G/N$ containing $I/N.$ Then $H/N$ does not contain the socle $M/N\cong\PSL_2(11)$ so $H/N\cong (C_{11}:C_5):C_2,\textrm{D}_{24},\Sym_4$ or $\textrm{D}_{20}$ with index $|G:H|=2^2\cdot 3,5\cdot 11,5\cdot 11$ or $2\cdot 3\cdot 11,$ respectively.

The first case cannot occur since $12\in\cd(G/N)$ and $H/N$ is nonabelian. If the last case holds, then $|\cd(G|\theta)|\leq 1.$

Assume that $H/N\cong \textrm{D}_{24}.$ If $\theta$ is $H$-invariant, then $\cd(G|\theta)=\{1,2\}\cdot 5\cdot 11\theta(1)$ or every degree in $\cd(G|\theta)$ is divisible by $2\cdot 5\cdot 11\theta(1).$ If the first case holds, then $\theta(1)=1$ and if the second case holds, then $|\cd(G|\theta)|=1.$  If $\theta$ is not $H$-invariant, then $|H:I|$ is divisible by $2$ or $3$ and hence every degree in $\cd(G|\theta)$ is divisible by $2\cdot 5\cdot 11$ or $3\cdot 5\cdot 11.$

Assume that $H/N\cong \Sym_4.$ Assume that $\theta$ is $H$-invariant. Then $\cd(G|\theta)=\{1,2,3\}5\cdot 11\theta(1)$ or $\{2,4\}5\cdot 11\theta(1)$. The latter case cannot occur and  if the first case holds, then $\theta(1)=1$.
If $\theta$ is not $H$-invariant, then $|H:I|$ is divisible by $2$ or $3$ so every degree in $\cd(G|\theta)$ is divisible by either $2\cdot 5\cdot 11$ or $3\cdot 5\cdot 11.$ Hence $|\cB_2|\le 4.$

(b) Assume $G/N\cong\PSL_2(11).$ Then $\cd(G/N)=\{1,5,10,11,12\}.$

Assume that $\theta\in\cA_1.$ Then $\theta$ is $G$-invariant and thus $\cd(G|\theta)=\{1,5,10,11,12\}\theta(1)$ or $\{6,10,12\}\theta(1).$
Since $12=2^2\cdot 3\in\cd(G/N),$ we deduce that $\theta(1)=1$ and thus $|\cB_1|\le 1.$

Assume that $\theta\in\cA_2.$ Then $I=I_G(\theta)\neq G$ and let $H/N$ be a maximal subgroup of $G/N$ containing $I/N.$ Then $H/N\cong\Alt_5,C_{11}:C_5$ or $\textrm{D}_{12}$ with index $|G:H|=11,2^2\cdot 3$ or $5\cdot 11.$ As above, the second case cannot happen.

Assume that $H/N\cong \textrm{D}_{12}.$ In this case $\cd(G|\theta)$ contains $5\cdot 11$ or a degree divisible by $2\cdot 5\cdot 11$ or $3\cdot 5\cdot 11$

Assume that $H/N\cong \Alt_5.$ Then $|G:H|=11.$

If $\theta$ is $H$-invariant, then $\cd(G|\theta)=\{1,3,4,5\}11\cdot\theta(1)$ or $\{2,4,6\}11\cdot\theta(1).$ If the latter case holds, then $\theta(1)=1;$ if the former case holds, then $\theta(1)=1$ or $\theta(1)$ is a prime and since $2^2\cdot 11$ divides some degree in $\cd(G|\theta)$, there is only one possibility for $\theta(1).$

Assume that $\theta$ is not $H$-invariant and let $K/N$ be a maximal subgroup of $H/N$ containing $I/N$. Then $K/N\cong\Alt_4,\Sym_3$ or $\textrm{D}_{10}$ with $|H:K|=5,10$ or $6$, respectively. If the last two cases holds, then $|\cd(G|\theta)|=1$ and the unique degree in $\cd(G|\theta)$ is divisible by either $2\cdot 3\cdot 11$ or $2\cdot 5\cdot 11.$

Assume $K/N\cong \Alt_4.$ If $I=K,$ then $\cd(G|\theta)=\{1,3\}5\cdot 11\cdot \theta(1)$ or $\{2\cdot 5\cdot 11\cdot \theta(1)\}$. If the first case holds, then $\theta(1)=1.$ If $I\neq K,$ then $|I:K|$ is divisible by either $2$ or $3$, so every degree in $\cd(G|\theta)$ is divisible by either $2\cdot 5\cdot 11$ or $3\cdot 5\cdot 11.$

Thus $|\cB_2|\le 12.$ Therefore, $|\cd(G)\le |\cB_1|+|\cB_2|+|\cd(G/N)|\le 12+1+5=18.$
\end{proof}

\begin{proof}[\textbf{Proof of Theorem \ref{th:L2q}}] If $q=7,8,9$ or $11$, then Theorem \ref{th:L2q} follows from Lemmas \ref{lem:L2(7)}, \ref{lem:L2(8)}, \ref{lem:L2(9)} and \ref{lem:L2(11)}, respectively.

Next, assume that $q>8$ is even. Then $|\cB_1|\le 4$ and $|\cB_2|\le 5$ by Lemmas \ref{lem:L2MInvariant} and \ref{lem:L2even}, respectively. As $|\cd(G/N)|\leq 6,$ we see that $|\cd(G)|\le |\cB_1|+|\cB_2|+|\cd(G/N)|\le 4+5+6<19.$

Finally, assume that $q\ge 13$ is odd. Then $|\cB_1|\le 4$ and $|\cB_2|\le 9$ by Lemmas \ref{lem:L2MInvariant} and \ref{lem:L2odd}, respectively. As $|\cd(G/N)|\leq 6,$ we see that $|\cd(G)|\le |\cB_1|+|\cB_2|+|\cd(G/N)|\le 4+9+6=19.$
\end{proof}

We now can prove Theorem \ref{th:thispaper}.

\begin{proof}[Proof of Theorem \ref{th:thispaper}]
We are assuming that $G$ satisfies the two-prime hypothesis and has a composition factor isomorphic to $\PSL_2 (q)$ for some  prime power $q \ge 7$.  In other words, there is a composition series for $G$ that has $\PSL_2 (q)$ as a composition factor for $G$.  Consider a chief series for $G$.  Recall that we may refine our chief series to a composition series for $G$.  By the Jordan-H\"older theorem, we know that all composition series for $G$ have the same composition factors.  Thus, we can find normal subgroups $L < K$ in our chief series so that $K/L$ is one of the chief factors and $L \le V < U \le K$ so that $U$ is subnormal in $K$ and $V$ is normal in $U$ satisfying $U/V \cong \PSL_2 (q)$.  Since $K/L$ is a chief factor, we know that $K/L$ is the direct product of copies of some simple group $S$.  The composition factors of $K/L$ will all be isomorphic to $S$ and since $U/V$ is a composition factor of $K/L$, we conclude that $S \cong \PSL_2 (q)$.  On the other hand, we know from Theorem \ref{chief} that all of the nonabelian chief factors of $G$ that are not isomorphic to $\Alt_5 \times \Alt_5$ are simple.  Since $q \ge 7$, we know that $\Alt_5 \not\cong \PSL_2 (q)$, so we conclude that $K/L$ is simple and in particular that $K/L$ is isomorphic to $\PSL_2 (q)$.

Let $C/L = C_{G/L} (K/L)$.  Observe that $C$ is normal in $G$ and $C \cap K = L$, so $CK/C \cong K/L$.  We know that $G/CL$ is isomorphic to a subgroup of ${\rm Aut} (K/L)$, and thus, we can conclude that $G/CL$ is an almost simple group.  In particular, $G$ now satisfies the hypotheses of Theorem \ref{th:L2q}, and the conclusion follows from there.
\end{proof}

We also can prove Theorem \ref{th:main}.

\begin{proof}[Proof of Theorem \ref{th:main}]
If $G$ has no composition factors isomorphic to $\PSL_2 (q)$ for any prime power $q$, then apply Theorem 1.1 of \cite{LLT}.  Otherwise, since $G$ is nonsolvable, and does not have any composition factors isomorphic to $\Alt_5$, then it must have some composition factor isomorphic to $\PSL_2 (q)$ for some prime power $q \ge 7$ and we apply Theorem \ref{th:thispaper} to obtain the conclusion.
\end{proof}

\end{document}